\documentclass[12pt,american]{article}

\usepackage{amsmath,amssymb,amsfonts}
\usepackage{amsthm}
\usepackage{booktabs}
\usepackage{url}
\usepackage{tikz}
\usepackage{pgfplots}
\usepackage{natbib}
\pgfplotsset{compat=1.18}

\usepackage{geometry}
\newtheorem{theorem}{Theorem}[section]
\newtheorem{lemma}[theorem]{Lemma}
\newtheorem{remark}[theorem]{Remark}

\newcommand{\R}{\mathbb{R}}
\newcommand{\E}{\mathbb{E}}
\newcommand{\Pp}{\mathbb{P}}
\newcommand{\cF}{\mathcal{F}}
\newcommand{\cB}{\mathcal{B}}
\newcommand{\cL}{\mathcal{L}}
\newcommand{\crit}{\operatorname{crit}}
\newcommand{\Acc}{\operatorname{Acc}}

\begin{document}

\title{A Mini-Batch Counterexample to Last-Iterate Convergence in Definable Optimization}
\author{Weiwei Kong}

\maketitle

\abstract{We give a counterexample to the convergence conjecture in Remark 12 of [Bolte \& Pauwels, 2021] for mini-batch stochastic approximation with definable potentials. The construction uses two convex piecewise-affine, hence semialgebraic, summands on $\mathbb{R}$. We choose a deterministic nonincreasing block stepsize sequence satisfying $\alpha_k = o(1/\log k)$ and an admissible minimum-norm selection from each aggregate batch field. On successive blocks, the iterates form lazy reflected random walks on nested dyadic lattices. An explicit endpoint-cover-time estimate, Markov’s inequality, and the first Borel-Cantelli lemma imply that almost surely every sufficiently late block's iterates visit their entire lattice. Consequently, the iterates remain in $[-1,1]$ but do not converge, and their accumulation set is exactly $[-1,1]$, on which the averaged objective is constant. Finally, the construction has $\sum_k \alpha_k^2 =\infty$. Both Chat-GPT 5.6 (Sol) and Gemini Pro 3.1 (DeepThink) were used in the development and drafting of this result.}
\section{Introduction}\label{sec:introduction}

Bolte and Pauwels \cite{BoltePauwels2021} study the finite-sum minimization problem
\begin{equation}\label{eq:objective}
    \min_{w\in\mathbb{R}^d} \left[J(w):=\frac1n \sum_{i=1}^n f_i(w)\right]
\end{equation}
and consider the iterates $\{w_k\}$ of the mini-batch stochastic approximation algorithm (SAA) generated by
\begin{equation}\label{eq:BP-recursion}
    w_{k+1}=w_k-\alpha_k d_k,
    \qquad
    d_k\in \frac1{|B_k|}\sum_{i\in B_k}D_i(w_k),
\end{equation}
where each $B_k\subset\{1,\dots,n\}$ is an independent uniformly sampled nonempty mini-batch, $f_i:\mathbb{R}^d \to \mathbb{R}$ is a Lipschitz-definable potential, and $D_i$ is a definable conservative field for $f_i$. Under boundedness and the stepsize condition
\begin{equation}\label{eq:BP-step}
    \alpha_k=o(1/\log k),
\end{equation}
they prove that the accumulation set is nonempty, critical, and contained in one level set of $J$. Remark~12 of \cite{BoltePauwels2021} conjectures convergence of the full sequence.

We give a one-dimensional counterexample consisting of two convex semialgebraic summands whose averaged objective is a deterministic hinge function with a flat critical interval. The stochasticity comes from sampling the finite-sum components, not from the averaged objective itself: in the interior of the plateau, the two component fields are $\{1\}$ and $\{-1\}$ and cancel only after averaging. 
A blockwise choice of stepsizes turns \eqref{eq:BP-recursion} into reflected random walks on finer and finer lattices. Almost surely, every sufficiently late block visits both endpoints and hence every point of its lattice. Since the lattice mesh tends to zero, the entire critical interval is the accumulation set.

Pathological bounded subgradient dynamics were already known for general Lipschitz objectives \cite{DaniilidisDrusvyatskiy2020}. R\'ios-Zertuche later showed that bounded deterministic subgradient sequences may fail to converge for a Whitney-stratifiable objective satisfying a Kurdyka--\L{}ojasiewicz inequality, and developed further pathologies for path-differentiable objectives \cite{RiosZertuche2022}. Subsequent work describes the structure of oscillations in deterministic and stochastic proximal subgradient schemes \cite{BoltePauwelsRios2024,Schechtman2023}. To our knowledge, these results do not furnish an explicit instance of the independent uniform finite-sum mini-batch recursion \eqref{eq:BP-recursion} with definable component potentials. Our contribution is such an instance, using two convex semialgebraic summands in one dimension.

\vspace*{1em}
\noindent\textit{Notation}. For a locally Lipschitz function $f:{\mathbb{R}}^d \to \mathbb{R}$, denote by
\[
\partial_C f(x) = \left\{ \xi \in \mathbb{R}^d : \langle \xi, v \rangle \le \limsup_{\substack{y \to x \\ t \downarrow 0}} \frac{f(y + tv) - f(y)}{t} \text{ for all } v \in \mathbb{R}^d \right\}
\]
its Clarke subdifferential. When $f$ is convex, this is the usual convex subdifferential and is a conservative field \cite{Clarke1990,BoltePauwels2021}. For any sequence $\{x_k\}\subseteq\R$, denote by
\begin{equation}\label{eq:acc-definition}
    \Acc(\{x_k\})
    :=\bigl\{x\in\R:\text{there exists }k_j\uparrow\infty
    \text{ such that }x_{k_j}\to x\bigr\}
\end{equation}
its set of accumulation points. For any nonempty closed convex set $S\subseteq\mathbb{R}^d$, denote by
\[
\operatorname{proj}_S(x)
:=\operatorname*{argmin}_{y\in S}\|x-y\|
\]
the (unique) Euclidean projection of $x$ onto $S$. For a set-valued field \(D:\mathbb{R}^p\rightrightarrows\mathbb{R}^p\), define its critical set by
\[
\crit(D):=\{x\in\mathbb{R}^p:0\in D(x)\}.
\]
When \(D_f\) is the conservative field associated with a function \(f\), we write
\[
\crit_f:=\crit(D_f).
\]

\vspace*{1em}
\noindent\textit{Organization}. Section~\ref{sec:construction} constructs the potentials, batch selection method, and states the main claim. Section~\ref{sec:proof} gives the proof of the main claim. Section~\ref{sec:discussion} gives a few concluding remarks.

\vspace*{1em}
\noindent\textit{AI Usage}. Both Chat-GPT 5.6 (Sol) and Gemini Pro 3.1 (DeepThink) were used in the development and drafting of this result.

\section{Construction and main result}\label{sec:construction}

This section defines the functions of interest, the sampling scheme and direction selector, and states the main result.

Define the key functions
\begin{equation}\label{eq:phi-fi}
    \phi(x):=\max\{|x|-1,0\},
    \qquad
    f_1(x):=\phi(x)+x,
    \qquad
    f_2(x):=\phi(x)-x,
\end{equation}
and set $D_i=\partial_C f_i$ for $i\in\{1,2\}$. The above functions are convex, piecewise-affine, and semialgebraic (hence, locally Lipschitz and definable). On the interval relevant to the construction, their subdifferentials $D_i(\cdot)$ are
\begin{equation}\label{eq:component-fields}
\begin{array}{c|ccc}
 & x=-1 & \quad-1<x<1\quad & x=1\\ \hline
D_1(x) & [0,1] & \{1\} & [1,2]\\
D_2(x) & [-2,-1] & \{-1\} & [-1,0]
\end{array}
\end{equation}
The average objective is
\begin{equation}\label{eq:Jphi}
    J(x)=\frac{f_1(x)+f_2(x)}2=\phi(x).
\end{equation}
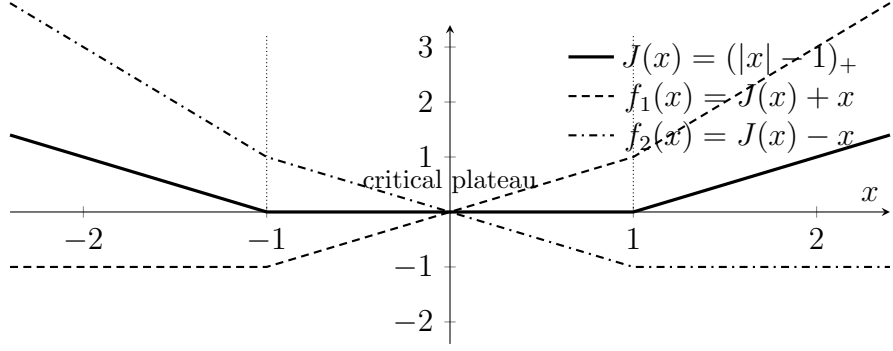
\begin{figure}[t]
\centering
\begin{tikzpicture}
\begin{axis}[
    width=0.80\textwidth,
    height=5.8cm,
    axis lines=middle,
    xlabel={$x$},
    xmin=-2.4, xmax=2.4,
    ymin=-2.4, ymax=3.4,
    domain=-2.4:2.4,
    samples=241,
    xtick={-2,-1,0,1,2},
    ytick={-2,-1,0,1,2,3},
    legend style={at={(0.98,0.98)},anchor=north east,draw=none,fill=none},
    clip=false
]
\addplot[very thick, black] {max(abs(x)-1,0)};
\addlegendentry{$J(x)=(|x|-1)_+$}
\addplot[thick, densely dashed] {max(abs(x)-1,0)+x};
\addlegendentry{$f_1(x)=J(x)+x$}
\addplot[thick, dashdotted] {max(abs(x)-1,0)-x};
\addlegendentry{$f_2(x)=J(x)-x$}
\draw[densely dotted] (axis cs:-1,0) -- (axis cs:-1,3.2);
\draw[densely dotted] (axis cs:1,0) -- (axis cs:1,3.2);
\node[above] at (axis cs:0,0.15) {\footnotesize critical plateau};
\end{axis}
\end{tikzpicture}
\caption{The deterministic objective $J$ is flat on $[-1,1]$, while its two component functions have opposite slopes there.}
\label{fig:hinge}
\end{figure}
Figure~\ref{fig:hinge} displays the functions $\phi$, $f_1$, and $f_2$.

Because the values in \eqref{eq:component-fields} are intervals, their Minkowski sum is already convex, and the averaged field of Bolte and Pauwels becomes
\begin{equation}\label{eq:DJ-explicit}
D_J(x)=\frac{D_1(x)+D_2(x)}{2}=\partial_C\phi(x)
=\begin{cases}
\{-1\},&x<-1,\\
[-1,0],&x=-1,\\
\{0\},&-1<x<1,\\
[0,1],&x=1,\\
\{1\},&x>1,
\end{cases}
\end{equation}
and, consequently,
\begin{equation}\label{eq:critinterval}
    \crit_J=[-1,1],
    \qquad
    J|_{[-1,1]}\equiv0.
\end{equation}

We now define the mini-batch sampling scheme and direction selector. Let the set of possible mini-batches be
\[
    \cB:=\bigl\{\{1\},\{2\},\{1,2\}\bigr\},
\]
and let the selected batches $\{B_k\}$ be independent and uniform on $\cB$. For $B\in\cB$, define
\begin{equation}\label{eq:selector}
    G_B(x)=\frac1{|B|}\sum_{i\in B}D_i(x),
    \qquad
    g_B(x)=\operatorname{proj}_{G_B(x)}(0),
    \qquad
    d_k=g_{B_k}(w_k),
\end{equation}
where $g_{B_k}(w_k)$ is the selected direction at iteration $k$.

Note that each $G_B(x)$ is a nonempty compact interval, so the projection in \eqref{eq:selector} is uniquely defined. If $G_B(x)=[\ell_B(x),u_B(x)]$, then
\[
    g_B(x)=\max\{\ell_B(x),\min\{0,u_B(x)\}\}.
\]
The endpoint functions $\ell_B,u_B$ are piecewise affine, so $g_B$ is single-valued, semialgebraic, and therefore measurable. By construction $g_B(x)\in G_B(x)$, so the oracle is admissible in \eqref{eq:BP-recursion}. Moreover, the selector is made from the aggregate batch field $G_B$, exactly as
permitted by the inclusion \eqref{eq:BP-recursion}; it need not arise
from batch-independent single-valued selections of the individual
fields $D_i$.

We are now ready to state the main claim.

\begin{theorem}\label{thm:main}
There is a deterministic nonincreasing positive sequence $\{\alpha_k\}$ satisfying \eqref{eq:BP-step} such that, for $w_0=0$, the recursion \eqref{eq:BP-recursion} with the selector \eqref{eq:selector} satisfies
\[
    |w_k|\le1\quad\text{for all }k,
    \qquad
    \Acc(\{w_k\})=[-1,1]=\crit_J\quad\text{almost surely}.
\]
Moreover, the iterates $\{w_k\}$ do not converge almost surely, even though $J$ is constant on their accumulation set.
\end{theorem}

\section{Proof of Theorem~\ref{thm:main}}\label{sec:proof}

This section gives the proof of Theorem~\ref{thm:main}. It additionally defines the stepsize sequence $\{\alpha_k\}$ and shows that the corresponding iterates $\{w_k\}$ form a scaled random walk when $w_0=0$.

We first define the stepsize sequence and show that the corresponding SAA iterates $\{w_k\}$ constitute a scaled lazy reflected random walk. For block $m\ge1$, define
\begin{equation}\label{eq:blockparameters}
    a_m:=2^{-m},
    \qquad
    N_m:=\frac{2}{a_m}=2^{m+1},
    \qquad
    L_m:=8mN_m^2,
\end{equation}
and let $K_1=0$ and $K_{m+1}=K_m+L_m$. We consider the stepsizes $\{\alpha_k\}$ given by
\begin{equation}\label{eq:steps}
    \alpha_k=a_m
    \quad\text{for }K_m\le k<K_{m+1}.
\end{equation}

\begin{lemma}\label{lem:block}
The sequence $\{\alpha_k\}$ given by \eqref{eq:steps} is positive, nonincreasing, and satisfies $\alpha_k=o(1/\log k)$. Moreover, during block $m$, the iterates $\{w_k\}$ given by \eqref{eq:BP-recursion} and \eqref{eq:selector} remain on the lattice
\[
    \cL_m=\{-1+ja_m:0\le j\le N_m\}\subset[-1,1],
\]
and the rescaled process
\begin{equation}\label{eq:Y}
    Y_t=\frac{w_{K_m+t}+1}{a_m},
    \qquad 0\le t\le L_m,
\end{equation}
with $w_0=0$ is a lazy symmetric random walk on nodes $\{0,\dots,N_m\}$ reflected at the endpoints.
\end{lemma}

\begin{proof}
We first prove that $\alpha_k=o(1/\log k)$. The sequence is positive and nonincreasing because $a_{m+1}=a_m/2$. Since $N_r^2=4^{r+1}$,
\[
    K_{m+1}=\sum_{r=1}^{m}8rN_r^2
    =32\sum_{r=1}^{m}r4^r
    \le C m4^m
\]
for an absolute constant $C$. If $k$ lies in block $m$, then $k<K_{m+1}$ and
\[
    \alpha_k\log(k+2)
    \le 2^{-m}\log(K_{m+1}+2)
    =O(m2^{-m})\longrightarrow0.
\]
Thus $\alpha_k=o(1/\log k)$.

We next prove lattice invariance by induction. Since $a_{m+1}=a_m/2$, every point of $\cL_m$ also belongs to $\cL_{m+1}$. The initial point $w_0=0$ lies in $\cL_1$. Suppose that $w_{K_m}\in\cL_m$. At an interior lattice point, the update rule \eqref{eq:BP-recursion} changes the lattice index by $-1$, $0$, or $1$ as $d_k\in\{-1,0,1\}$ for $w_k\in[-1,1]$. At $x=-1$, the directions corresponding to
$B=\{1\},\{2\},\{1,2\}$ are respectively $0,-1,0$.
Thus the iterate stays at $-1$ with probability $2/3$ and moves one
lattice step inward with probability $1/3$.
Similarly, at $x=1$, the three directions are respectively $1,0,0$,
so the iterate stays at $1$ with probability $2/3$ and moves one
lattice step inward with probability $1/3$. Hence, all iterates in block $m$, including $w_{K_{m+1}}$, remain in $\cL_m\subseteq[-1,1]$. The nesting of the lattices closes the induction at the next block.

Finally, we prove that $\{Y_t\}$ is the claimed random walk. First, note that at each iteration, the three possible mini-batches $\{1\}$, $\{2\}$, and $\{1,2\}$ are equiprobable, and the batch selections are independent across iterations. Therefore, for $1\le j\le N_m-1$, we have
\[
\Pp(Y_{t+1}=j-1\mid Y_t=j)
=\Pp(Y_{t+1}=j\mid Y_t=j)
=\Pp(Y_{t+1}=j+1\mid Y_t=j)=\frac13.
\]
At $j=0$, the probabilities of staying and moving to $1$ are $2/3$ and $1/3$, respectively; at $j=N_m$ the analogous probabilities are $2/3$ and $1/3$ toward $N_m-1$. This is precisely the lazy symmetric random walk reflected at the endpoints.
\end{proof}

\begin{remark}\label{rem:stepsums}
The stepsizes also satisfy
\[
\sum_{k=K_m}^{K_{m+1}-1}\alpha_k=L_ma_m=32m2^m,
\qquad
\sum_{k=K_m}^{K_{m+1}-1}\alpha_k^2=L_ma_m^2=32m.
\]
Consequently, both $\sum_k\alpha_k$ and $\sum_k\alpha_k^2$ diverge. The second divergence is the source of persistent sampling noise and is permitted by \eqref{eq:BP-step}.
\end{remark}

We now show that the expected time for the above random walk on nodes $\{0,\ldots,N\}$ to visit both endpoints is at most $4N^2$.

\begin{lemma}[Endpoint cover time]\label{lem:cover}
Let $\{Y_t\}$ be the lazy symmetric random walk in Lemma~\ref{lem:block} on nodes $\{0,\dots,N\}$. The expected additional time needed to visit both endpoints is at most $4N^2$, uniformly over the starting state and which endpoints have already been visited.
\end{lemma}

\begin{proof}
Throughout this proof, let $\E_j[X]=\E[X\mid Y_0=j]$ denote the expectation for the reflected chain started from $Y_0=j$
for every integrable random variable $X$ associated with the chain.

First suppose that neither endpoint has yet been visited. Let $H_{\partial}$ be the first hitting time of $\{0,N\}$ and denote $h_j=\E_j[H_{\partial}]$. Conditioning on the first step gives, for $1\le j\le N-1$,
\[
    h_j=1+\frac13h_{j-1}+\frac13h_j+\frac13h_{j+1},
    \qquad h_0=h_N=0,
\]
or, equivalently,
\[
    h_{j+1}-2h_j+h_{j-1}=-3, \quad h_0 = h_N = 0.
\]
Solving the inhomogeneous recurrence gives the unique solution $h_j= 3j(N-j)/2$ and, in particular, we have $\max_j h_j\le3N^2/8$.

If the first endpoint is reached, it remains to hit the opposite one. Let $H_N$ be the first hitting time of $N$ for the chain with $0$ reflecting, and denote $q_j=\E_j[H_N]$. Then, $q_N=0$, while at the reflecting endpoint we have
\[
    q_0=1+\frac23q_0+\frac13q_1
    \quad\text{or, equivalently,}\quad q_0-q_1=3.
\]
Solving the inhomogeneous recurrence gives the unique solution
\begin{equation}\label{eq:hitting-opposite}
    q_j=\frac32\left[N(N+1)-j(j+1)\right]
\end{equation}
and, in particular, we have $\max_jq_j=q_0=3N(N+1)/2\le3N^2$ 
for $N\ge1$. By symmetry, the same estimate holds for hitting $0$ when $N$ is reflecting.

We now combine the two estimates. Let $T_{\mathrm{cov}}$ be the additional time needed to visit both endpoints when neither has yet been visited, and write
\[
    T_{\mathrm{cov}}=H_{\partial}+R,
\]
where $R$ is the time, after $H_{\partial}$, needed to reach the endpoint opposite to $Y_{H_{\partial}}$. Let $\mathcal H_{H_{\partial}}$ denote the information accumulated by the chain up to the stopping time $H_{\partial}$. The strong Markov property says that, conditionally on this information, the post-$H_{\partial}$ process is a fresh copy of the same chain started from the random state $Y_{H_{\partial}}\in\{0,N\}$. Therefore
\[
    \E_j[R\mid\mathcal H_{H_{\partial}}]
    =\mathbf 1_{\{Y_{H_{\partial}}=0\}}\E_0[H_N]
     +\mathbf 1_{\{Y_{H_{\partial}}=N\}}\E_N[H_0]
    \le3N^2.
\]
Using the tower property, we obtain
\[
\begin{aligned}
    \E_j[T_{\mathrm{cov}}]
    &=\E_j[H_{\partial}]
      +\E_j\!\left[\E_j[R\mid\mathcal H_{H_{\partial}}]\right] \le \frac38N^2+3N^2
    <4N^2.
\end{aligned}
\]
If one endpoint has already been visited, only the endpoint-to-endpoint term is needed; if both have been visited, the additional time is zero. The bound is therefore uniform over the augmented state consisting of the current location and the set of endpoints already visited.
\end{proof}

For the remaining arguments, let $\{w_k\}$ denote the iterate sequence
generated by \eqref{eq:BP-recursion}, \eqref{eq:selector}, and the
stepsizes \eqref{eq:steps}. Moreover, let $\cF_0$ be the trivial $\sigma$-field and define
\begin{equation}\label{eq:filtration}
    \cF_k:=\sigma(B_0,\dots,B_{k-1}),
    \qquad k\ge1.
\end{equation}
Then, $w_k$ is $\cF_k$-measurable, while the mini-batches from time $k$ onward remain independent and uniform on $\cB$ conditionally on $\cF_k$.

The next lemma bounds the block failure probability.

\begin{lemma}[Block success probability]\label{lem:block-success}
Let $A_m$ be the event that the iterates $\{w_k\}$ in block $m$ fail to visit both endpoints $-1$ and $1$. Then,
\begin{equation}\label{eq:block-failure}
    \Pp(A_m\mid\cF_{K_m})\le2^{-m}
    \qquad\text{almost surely}.
\end{equation}
Consequently, $\sum_{m\ge1}\Pp(A_m)<\infty$.
\end{lemma}

\begin{proof}
Fix $m$, write $N=N_m$, and divide the block into $m$ epochs of length $\ell:=8N^2$. 
For $r=0,\dots,m$, set
\[
    t_r:=K_m+r\ell,
    \qquad
    \mathcal G_r:=\cF_{t_r},
\]
and let $E_r$ be the event that the endpoint cover is incomplete by
time $t_r$, that is, that the iterates $w_k$ with
$K_m\le k\le t_r$ have not visited both endpoints.
Then $E_0$ occurs surely and $A_m=E_m$.

Fix $r\in\{1,\dots,m\}$. Conditional on $\mathcal G_{r-1}$,
the next epoch is driven by fresh independent uniform mini-batches.
Let $V_{r-1}\subseteq\{0,N\}$ be the set of endpoints visited by the
rescaled process by time $t_{r-1}$. Define $T_r$ to be the additional
number of steps needed to visit both endpoints, starting from the
augmented state consisting of the current position and $V_{r-1}$.
Thus $T_r=0$ if both endpoints have already been visited.

Since the epoch has length $8N^2$,
\[
E_r=E_{r-1}\cap\{T_r>8N^2\}.
\]
By the conditional form of Markov's inequality and
Lemma~\ref{lem:cover}, we have
\[
\begin{aligned}
\Pp(E_r\mid\mathcal G_{r-1})
&=\mathbf 1_{E_{r-1}}
  \Pp(T_r>8N^2\mid\mathcal G_{r-1})\\
&\le
\mathbf 1_{E_{r-1}}
\frac{\E[T_r\mid\mathcal G_{r-1}]}{8N^2}
\le 
\mathbf 1_{E_{r-1}}
\frac{4N^2}{8N^2}
= \frac12\mathbf 1_{E_{r-1}}.
\end{aligned}
\]
Because $\cF_{K_m}\subseteq\mathcal G_{r-1}$, the law of total expectations and the above bound yields
\begin{align*}
    \Pp(E_r\mid\cF_{K_m})
    &=\E\!\left[
       \Pp(E_r\mid\mathcal G_{r-1})
       \,\middle|\,\cF_{K_m}\right] 
       \le\frac12\Pp(E_{r-1}\mid\cF_{K_m}).
\end{align*}
Iterating from $\Pp(E_0\mid\cF_{K_m})=1$ proves \eqref{eq:block-failure}. Taking expectations gives $\Pp(A_m)\le2^{-m}$, and the claimed summability follows.
\end{proof}

The next lemma describes the accumulation set when every sufficiently
late block visits both endpoints.

\begin{lemma}[Accumulation from successful blocks]\label{lem:accumulation}
Suppose that the iterates $\{w_k\}$ for all sufficiently large blocks $m\geq 1$ visit both endpoints. Then, $\Acc(\{w_k\})=[-1,1]$.
\end{lemma}

\begin{proof}
In a successful block $m\geq1$, the rescaled process $\{Y_t\}$ in Lemma~\ref{lem:block} is integer-valued, has increments of magnitude at most one, and visits both $0$ and $N_m$. It must therefore visit every state in $\{0,\dots,N_m\}$, so the original process visits every point of $\cL_m$.

Fix $x\in[-1,1]$. For each sufficiently large successful block, choose $z_m\in\cL_m$ with
\begin{equation}\label{eq:lattice-approximation}
    |z_m-x|\le\frac{a_m}{2}.
\end{equation}
In each selected block choose an index $k_m\in[K_m,K_{m+1}]$ for which $w_{k_m}=z_m$. Since $k_m\ge K_m$ and $K_m\to\infty$, we have $k_m\to\infty$; after passing to a strictly increasing subsequence if necessary, \eqref{eq:lattice-approximation} and $a_m\to0$ give
\[
    w_{k_m}=z_m\longrightarrow x.
\]
Thus every $x\in[-1,1]$ is an accumulation point on the same trajectory. Conversely, Lemma~\ref{lem:block} gives $w_k\in[-1,1]$ for every $k$, and the interval is closed. Hence, every accumulation point belongs to $[-1,1]$.
\end{proof}

We are now ready to give the proof of Theorem~\ref{thm:main}.

\begin{proof}[Proof of Theorem~\ref{thm:main}]
The functions $f_1,f_2$ are convex and semialgebraic, the fields $D_i=\partial_C f_i$ are definable conservative fields, and the selector \eqref{eq:selector} is measurable and admissible. Lemma~\ref{lem:block} shows that the stepsizes are positive, nonincreasing, and satisfy \eqref{eq:BP-step}, and that $w_k\in[-1,1]$ pathwise.

By Lemma~\ref{lem:block-success}, we have $\sum_{m\ge1}\Pp(A_m)<\infty$, and the first Borel-Cantelli lemma therefore implies that, almost surely, only finitely many blocks fail to visit both endpoints. Consequently, on the probability-one event that there is a block $m_0$ such that all blocks $m\geq m_0$ contain iterates that visit the endpoints $-1$ and $1$, Lemma~\ref{lem:accumulation} gives $\Acc(\{w_k\})=[-1,1]$ and 
that the sequence does not converge. Finally, \eqref{eq:critinterval} shows that $[-1,1]=\crit_J$ and that $J$ vanishes on this set. 
\end{proof}

\section{Concluding remarks}\label{sec:discussion}

The full-batch minimum-norm direction vanishes throughout the flat critical interval, whereas the singleton-batch directions are nonzero in its interior. Definability controls the geometry of the plateau without suppressing tangential sampling noise. The boundary subgradients reflect the process and keep it bounded.

This paper does not claim the first nonconvergent subgradient dynamics under geometric regularity; the examples in \cite{DaniilidisDrusvyatskiy2020,RiosZertuche2022} are substantially broader in that direction. The point here is narrower: the conjecture fails inside the exact finite-sum mini-batch model in which it was posed, already for two convex semialgebraic summands and a explicit minimum-norm selector.

As recorded in Remark~\ref{rem:stepsums}, the construction has $\sum_k\alpha_k^2=\infty$, which is allowed by \eqref{eq:BP-step}. It therefore does not rule out stronger convergence results under square-summable steps, noise that vanishes on the critical set, or assumptions forcing the relevant critical component to be a singleton.
\section*{Statements and Declarations}

This manuscript was prepared with assistance from ChatGPT 5.6 (Sol) and Gemini 3.1 Pro (DeepThink) for brainstorming, language editing, and LaTeX assistance. The author independently verified all mathematical arguments, references, and claims and takes full responsibility for the manuscript.

\bibliographystyle{plain}
\bibliography{references}

\end{document}